\documentclass[12pt,reqno]{amsart}

\usepackage[english]{babel}
\usepackage[utf8]{inputenc}
\usepackage{amsmath}
\usepackage{amssymb}
\usepackage{amsthm}
\usepackage{enumerate}
\usepackage{graphicx}
\usepackage{fullpage}
\usepackage{cite}
\newtheorem{thm}{Theorem}[section]
\newtheorem{lemma}[thm]{Lemma}
\newtheorem{cor}[thm]{Corollary}

\newtheorem{defn}[thm]{Definition}

\newcommand{\ZZ}{\mathbb{Z}}
\newcommand{\B}{\mathcal B}
\newcommand{\C}{\mathcal C}
\DeclareMathOperator{\Ker}{Ker}
\DeclareMathOperator{\Img}{Im}

\title{The sandpile group of polygon rings and twisted polygon rings}

\author{Haiyan Chen}
\address{%
The School of Sciences\\
Jimei University\\
Fujian, China}
\email{chey5@jmu.edu.cn}

\author{Bojan Mohar}
\address{%
Department of Mathematics\\
Simon Fraser University\\
Burnaby, BC, Canada}
\email{mohar@sfu.ca}

\thanks{This work was done while the first author visited The Simon Fraser University. The hospitality of the hosting institution is greatly acknowledged. The visit was funded by the Fujian Provincial Education Department.}
\thanks{H.~Y.~Chen was supported by the National Natural Science Foundation of China (grant numbers
11771181, 12071180).}
\thanks{B.M.~was supported in part by the NSERC Discovery Grant R611450 (Canada), by the Canada Research Chairs program, and by the Research Project J1-8130 of ARRS (Slovenia).}
\thanks{On leave from IMFM, Department of Mathematics, University of Ljubljana.}%

\date{\today}

\begin{document}

\begin{abstract}
Let $C_{k_1}, \ldots, C_{k_n}$  be cycles with $k_i\geq 2$ vertices ($1\le i\le n$). By attaching these $n$ cycles together in a linear order, we obtain a graph called a polygon chain. By attaching these $n$ cycles together in a cyclic order, we obtain a graph, which is called a polygon ring if it can be embedded on the plane; and called a twisted polygon ring if it can be embedded on the M\"{o}bius band.  It is known that the sandpile group  of a polygon chain is always cyclic. Furthermore, there exist edge generators. In this paper, we not only show that the sandpile group of any (twisted) polygon ring can be generated by at most three edges, but also  give an explicit relation matrix among these edges. So we obtain a uniform method to compute the sandpile group  of arbitrary (twisted) polygon rings, as well as the number of spanning trees of (twisted) polygon rings. As an application, we compute the sandpile groups of several infinite families of polygon rings, including some that have been done before by ad hoc methods, such as, generalized wheel graphs, ladders and M\"{o}bius ladders.
\end{abstract}

\maketitle

\section{Introduction}

The abelian sandpile models were firstly introduced in 1987 by three physicists, Bak, Tang,
and Wiesenfeld \cite{Bak1987Self}, who studied it mainly on the integer grid graphs. In 1990, Dhar\cite{Dhar1990Self} generalized their model from grids to arbitrary graphs.
The abelian sandpile model of Dhar begins with a
connected graph $G = (V, E)$ and a distinguished vertex $q\in V $, called the \emph{sink}. A \emph{configuration} of $(G, q)$ is a vector $\vec{c}\in \mathbb{N}^{V-q}$. A non-sink vertex $v$ is \emph{stable} if its degree satisfies $d(v) > \vec{c}(v)$; otherwise it is \emph{unstable}. Moreover, a configuration is \emph{stable}
if every vertex $v$ in $V-q$ is stable. \emph{Toppling} an unstable vertex $u\in V-q$ in $\vec{c}$ is the operation performed by decreasing
$\vec{c}(u)$ by the degree $d(u)$, and for each neighbour $v$ of $u$ different from $q$, adding the multiplicity $m(u,v)$ of the edge $uv$ to $\vec{c}(v)$. Starting from any initial configuration $\vec{c}$, by performing a sequence of topplings, we eventually arrive at a stable configuration. It is not hard to see that the stabilization of an unstable configuration is
unique \cite{Dhar1990Self,Biggs1999}. The stable configuration associated to $\vec{c}$ will be denoted by $s(\vec{c})$.
Now, let $(\vec{c} + \vec{d})(u):= \vec{c}(u) + \vec{d}(u)$ for all $u\in V-q$ and $\vec{c} \oplus \vec{d }:= s(\vec{c} + \vec{d})$. A configuration $\vec{c}$ is \emph{recurrent}
if it is stable and there exists a non-zero configuration $\vec{r}$ such that $s(\vec{c}+\vec{r}) = \vec{c}$. Dhar\cite{Dhar1990Self} proved that the number of recurrent configurations is equal to the number of spanning trees of $G$, and that the set of recurrent configurations with $\oplus$ as a binary operation forms a finite abelian group, which is called the \emph{sandpile group} of $G$.
Soon after that, it was found that the sandpile group is isomorphic
to a number of `classical' abelian groups associated with graphs, such as the group of components in Arithmetic Geometry \cite{Lorenzini1989Arithmetical}, Jacobian group and Picard group in Algebraic Geometry\cite{Bak2007Riemann}, the
determinant group in lattice theory\cite{Bacher1997The},
the critical group of a dollar game \cite{Biggs1997Algebraic,Biggs1999}.

As an abstract abelian group, the structure of the sandpile group is independent of the choice of the sink $q$. We denote the sandpile group of $G$ by $S(G)$. The classification theorem for finite abelian groups asserts that $S(G)$ has a direct sum
decomposition 
$$S(G)=\mathbb{Z}_{d_1}\oplus\mathbb{Z}_{d_2}\oplus\cdots \oplus\mathbb{Z}_{d_r},$$
where the integers $d_1,\ldots, d_r$, called the \emph{invariant factors} of $S(G)$, satisfy $d_1\ge2$ and
$d_i \mid d_{i+1}$ for $i=1,\dots,r-1$. The number of factors $r$ is the minimum number of generators of $S(G)$, and is denoted by $\mu(G)$.

The standard method of obtaining invariant factors of a finite abelian group is first to
choose a presentation of the group, and then compute the Smith Normal From (SNF) of the matrix of
relations. For the sandpile group, it is well known that the Laplacian matrix is one of its relation matrices. By computing the SNF of the Laplacian matrix, the sandpile groups for many special families of regular or near regular graphs have been completely or partially determined in the last twenty
years, such as wheel graphs, complete multipartite graphs, Cartesian product of complete graphs $K_n\square K_m$, the ladder $K_2\square C_n$, the M\"{o}bius
ladder graph, a Cayley graph $D_n$ of dihedral group, the squared cycle $C^2_n$, the thick cycle, etc.; see
\cite{Chen2008On,Chen2006On,Hou2006On,Christianson2002The,Deryagina2014On,Krepkiy2013The, Jacobson2003Critical,Musiker2009The,Shi2011The,Toumpakari2007On,
Glass2014Critical,Reiner2014Critical,Alfaro2012On,Chandler2015The,Berget2012The, Chan2014Sandpile,Ducey2018The,Ducey2017On,Raza2015On,Bai2003On,Glass2017Critical,
Bond2016The,Levine2011Sandpile,GOEL2019138,NCG00966} and references therein.

However, for families of irregular graphs, the SNF of the Laplacian matrix is not easy to get in general. So in \cite{CHEN201968}, the authors of the present paper determined the sandpile group of a family of irregular graphs by using the relations determined by the cycles and cuts of the graph. Using this method, it is easy to show that the sandpile group of a polygon chain is always cyclic, which has been proved earlier in \cite{Becker2016Cyclic,Krepkiy2013The}. Here, we shall follow this idea to study the sandpile group of a (twisted) polygon ring.  We first show that the sandpile group of any (twisted) polygon ring can be generated by at most three edges, and then we give an explicit relation matrix among these edges. So the sandpile group  of any (twisted) polygon ring is the direct sum of at most three cyclic groups. Our result can be used for computing sandpile groups of various infinite families of (twisted) polygon rings, by computing the Smith Normal Form of the relation matrix. Examples include many known families, such as generalized wheel graphs, ladder graphs, M\"{o}bius ladders, etc.\cite{Chen2006On,Sandpile2003,Deryagina2014On}.

The paper is organized as follows. In Section 2, we cover preliminaries. In Section 3, by using the structure properties of (twisted) polygon rings, we not only show that the sandpile group of any polygon ring can be generated by at most three edges (see Theorem \ref{thm:3.1}), but also give a relation matrix for these generators (see Theorem 3.2).  This is  the most valuable aspect of this study, since these results provide a uniform method to compute the sandpile group of any (twisted) polygon ring. As applications, we give in Section 4 the explicit relation matrices for a special family of polygon rings. In Sections 5 and 6, we determine the structure of sandpile groups by computing the Smith Normal Form of the matrices given in Section 4  for polygon rings and twisted polygon rings, respectively.

\section{Preliminaries}

Let $G=(V,E)$ be a connected graph with $n$ vertices and $m$ edges. Given an arbitrary orientation $\mathcal{O}$ of $E$, and an oriented edge $e=(u,v)$, $v$ is called the \emph{head} of $e$, denoted by $h(e)$, and $u$ is called the \emph{\emph{tail}} of $e$, denoted by $t(e)$. As the convention, if $e=(u,v),$ then $-e=(v,u)$.  Let $\mathbb{Z}V, \mathbb{Z}E$ denote the free abelian groups on $V$ and $E$, respectively. More clearly, every element $x\in \mathbb{Z}V $ is identified with the formal sum $\sum_{v\in V(G)}x(v)v$, where $x(v)\in \mathbb{Z}$, and similarly for $y\in\mathbb{Z}E$.

Consider a cycle $C= v_1e_1v_2e_2\cdots v_{k}e_kv_1$ in the undirected
graph $G$. The sign of an edge $e$ in $C$ with respect to the orientation $\mathcal{O}$ is $\sigma(e;C) = 1$
if $C = vev$ is a loop at the vertex $v$, and otherwise
$$
\sigma(e;C)=\left\{
\begin{aligned}
1,&\ \ \ \mbox{if}\ \ e\in C \  \mbox{and}\ \ t(e)=v_i, h(e)=v_{i+1}\ \  \mbox {for some}\ \  i;\\
-1,&\ \ \ \mbox{if}\ \ e\in C\ \  \mbox{and}\ \  t(e)=v_{i+1}, h(e)=v_i\ \  \mbox {for some}\ \  i;\\
0,&\ \ \ \mbox{otherwise\ ($e$ does not occur in $C$)}.
\end{aligned}
\right.
$$
Here we interpret indices module $k$, i.e., $v_{k+1}=v_1$.
We then identify $C$ with the formal sum
$\sum_{e\in E}\sigma(e,C)e\in \mathbb{Z}E$.

For each nonempty $U\subset V $, the
cut corresponding to $U$, denoted by $c_U$, is the collection of edges with one end vertex
in $U$ and the other in the complement $\overline{U}$. For each $e\in E$, define the sign of $e$ in
$c_U$ with respect to the orientation $\mathcal{O}$ by

$$
\sigma(e;c_U)=\left\{
\begin{aligned}
1,\ \ \  &\mbox{if}\ \ t(e)\in U \ \ \mbox{and}\ \  h(e)\in \overline{U};\\
-1,\ \ \  &\mbox{if}\ \  t(e)\in \overline{U}\ \ \mbox{and}\ \  h(e)\in U;\\
0,\ \ \  &\mbox{otherwise\ ($e$ does not occur in $c_U$)}.
\end{aligned}
\right.
$$
We then identify $c_U$ with the formal sum $\sum_{e\in E}\sigma(e,c_U)e\in \mathbb{Z}E$.

A vertex cut is the cut corresponding to a single vertex, $U =\{v\}$,
and we write $c_v$ for $c_U$ in this case.

\begin{defn}
The (integral) \emph{cycle space}, $\mathcal{C}\subseteq \mathbb{Z}E$, is the $\mathbb{Z}$-span of all cycles. The (integral) \emph{cut space}, $\mathcal{B}\subseteq \mathbb{Z}E$, is the $\mathbb{Z}$-span of all cuts.
\end{defn}

Let $L(G)$ be the Laplacian matrix of $G$. It can be viewed as a (linear) mapping $\mathcal{L}:$ $\mathbb{Z}V\rightarrow\mathbb{Z}V$. We also define a mapping $\rho:$ $\mathbb{Z}V\rightarrow\mathbb{Z}$ as $\rho(\sum_{v\in V}x(v)v)=\sum_{v\in V}x(v).$

Obviously, both $\mathcal{L}$ and $\rho$ are group homomorphisms. Then
we have the following well-known results.

\begin{thm}
Let $G=(V,E)$ be a graph. With the notation defined above, we have
$$S(G)\cong \frac{\Ker(\rho)}{\Img(\mathcal{L})}\cong\frac{\mathbb{Z}E}{\mathcal{C}\oplus \mathcal{B}}\, ,$$
where $\Ker(.)$ and $\Img(.)$ denote the kernel and the image of a mapping.
\end{thm}

The middle presentation of the sandpile group in Theorem 2.2 is the well-known Jacobian group (also known as Picard group) of the graph. The Jacobian presentation has a natural set of generators for $S(G)$, for which the Laplacian matrix $L(G)$ of $G$ is a relation matrix. For more details, see \cite{Biggs1999}.

Here we mainly focus on the second presentation. For any $e\in E$, let $\delta_e = \sum_{f\in E}\delta_e(f)f\in \mathbb{Z}E$, where $\delta_e(f)=1$ if $e=f$ and $0$ otherwise. Then the collection of all $\delta_{e}$ is
a natural set of generators of the sandpile group $S(G)$, and the relations are given by the elements in $\mathcal{C}\oplus \mathcal{B}$. So to find a relation matrix, we only need to find a basis of the cycle space $\mathcal{C}$ and a basis of the cut space $\mathcal{B}$, respectively.

By considering an edge $e$ of the graph as the element $\delta_e$ of $\mathbb{Z}E / (\mathcal{C}\oplus \mathcal{B})$, we consider the edges as particular elements of the sandpile group. We say that a set $F$ of edges \emph{generates} the sandpile group if $S(G)$ is generated by the set $\{\delta_e\mid e\in F\}$. 

Now we recall the definition and basic properties of the Smith Normal Form (SNF) of an integer matrix.
Let $M, N$ be two $n\times n$ integer matrices. The two matrices are called \emph{equivalent} if there exist invertible integer matrices $P$ and $Q$ (i.e., $|\det(P)|=|\det(Q)|=1$) such that $PMQ=N$. We have the following well-known results.

\begin{thm}\label{th:1}
{\rm (1)} Each integer matrix $M$ with rank $r$, is equivalent to a diagonal matrix $diag(d_1,\ldots,d_r, 0, \ldots, 0)$, where $d_i\,|\,d_{i+1}$, $i=1,\ldots,r-1$, and all these integers are positive. Furthermore, the $d_i$ are uniquely determined by
$$d_i=\frac{\Delta_i}{\Delta_{i-1}},\  i=1, \ldots, r,$$
where $\Delta_i$ (called $i$-th \emph{determinant divisor}) equals the greatest common divisor of all $i\times i$ minors of the matrix $M$ ($1\le i\le r$) and $\Delta_{0}=1$.

{\rm (2)} Let $A$ be a finite abelian group with presentation $A=\{g_1,\ldots,g_n\ |\ \sum_{j=1}^n m_{ij}g_j=0, i=1,\ldots,n\} $. If $M=(m_{ij})$ is equivalent to the diagonal matrix $diag(d_1,\ldots,d_r, 0, \ldots, 0)$ then
$$A\cong\mathbb{Z}_{d_1}\oplus\cdots\oplus \mathbb{Z}_{d_r}.$$
\end{thm}

The diagonal matrix in Theorem 2.3 (1) is called the \emph{Smith Normal Form} of $M$, and the integers $d_i$ are called \emph{invariant factors} of $M$. The matrix $M$ related to a presentation of the abelian group $A$ in part (2) of Theorem 2.3 is called the \emph{relation matrix} of $A$.

From (1) of the above theorem, we see that equivalent matrices have the same invariant factors. And (2) says that the invariant factors of $A$ are just the non-trivial invariant factors (those that are $\geq 2$) of its arbitrary relation matrix. So, to determine the structure of a finite abelian group, it is sufficient to find a set of generators and a complete set of relations among them, then compute the Smith Normal Form of the corresponding relation matrix.  In this paper, we shall start from the natural set of generators $\delta_e$ $(e\in E(G))$ to study the sandpile groups of polygon rings.

Let $a_1,\dots,a_n$ and $b_1,\dots,b_n$ be sequences of non-negative integers, and let $a=\sum_{i=1}^n a_i$ and $b=\sum_{i=1}^n b_i$. Take two paths, $Q^1 = x_0x_1\dots x_a$ of length $a$ and $Q^2 = y_0y_1\dots y_b$ of length $b$. 
For $j=0,1,\dots n$, define the vertex $v_i$ to be the vertex $x_t$, where $t=\sum_{i=1}^j a_i$. Similarly, define the vertex $u_i$ to be the vertex $y_t$, where $t=\sum_{i=1}^j b_i$. In particular, $v_0=x_0$ and $v_n = x_a$. We now define the \emph{polygon chain} $G_n = G_n(a_1\dots a_n; b_1\dots b_n)$ as the graph obtained from the union $Q^1\cup Q^2$ by adding edges $v_iu_i$ for $i=0,1,\dots,n$. 

Let $k_i = a_i+b_i+2$. The polygon chain defined above consists of a sequence of polygons with $k_1, \ldots , k_n$ sides. The first one of these polygons is a cycle $C_{k_1}$ of length $k_1$ and the last one is the cycle $C_{k_n}$. These two are called \emph{end-polygons}. The edges $v_0u_0$ and $v_nu_n$ in the end-polygons are called \emph{free edges} of $G_n$ (see Figure 1). By identifying $v_0$ with $v_n$, $u_0$ with $u_n$, and thus identifying $v_0u_0$ with $v_nu_n$, we obtain a \emph{polygon ring}, denoted by $R_n$ (see Figure 2(a)). If we identify $v_0$ with $u_n$, $u_0$ with $v_n$ we obtain a \emph{twisted polygon ring} $T_n$ (see Figure 3(a)).

\begin{figure}[htbp]
\centering
\scalebox{1.15}{\includegraphics{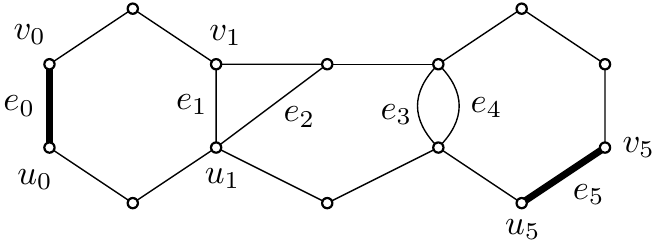}}
\caption{Polygon chain $G_5(2\,1\,1\,0\,3; 2\,0\,2\,0\,1)$ with its free edges $e_0,e_5$ drawn thick.}
\end{figure}

\begin{figure}[htbp]
\centering
\scalebox{1}{\includegraphics{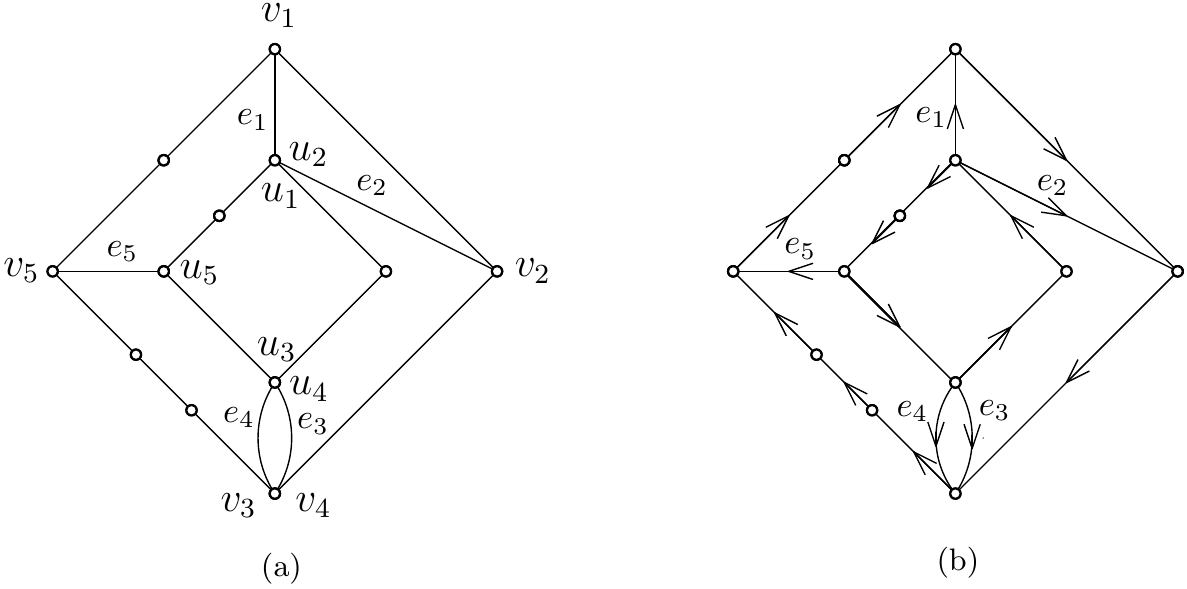}}
\caption{(a) The polygon ring $R_5(2\,1\,1\,0\,3; 2\,0\,2\,0\,1)$ obtained from $G_5$ in Figure 1 by identifying $e_0$ and $e_5$. (b) The oriented polygon ring $R_5$.}
\end{figure}

\begin{figure}[htbp]
\centering
\scalebox{1}{\includegraphics{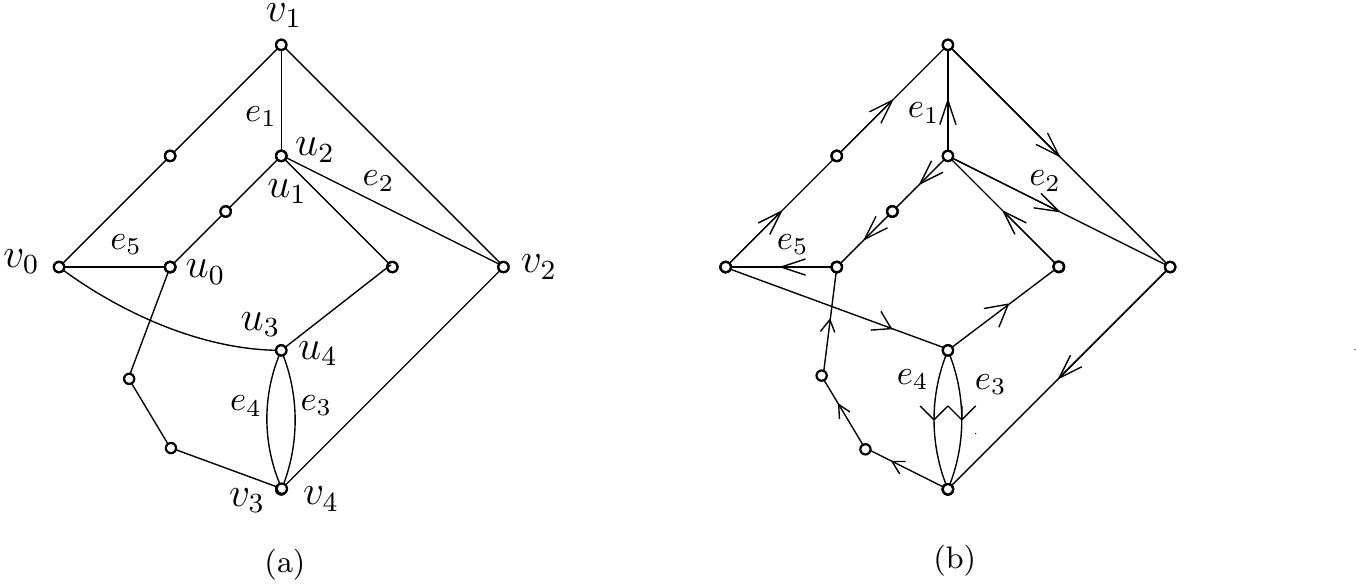}}
\caption{(a) The twisted polygon ring $T_5(2\,1\,1\,0\,3; 2\,0\,2\,0\,1)$ obtained from $G_5$ in Figure 1. (b) The oriented twisted polygon ring $T_5$.}
\end{figure}

As we mentioned before, the sandpile group of any polygon chain is cyclic. Moreover, the sandpile group $S(G_n)$ only depends on $(k_1, \ldots, k_n)$, and it is independent of the way that the polygons stack together \cite{Becker2016Cyclic,Krepkiy2013The}.  We shall prove that the sandpile group of a (twisted) polygon ring $R_n (T_n)$ is always the direct sum of at most three cyclic groups by showing it can be generated by at most three edges. However, it should be noted that,  unlike the polygon chain, the structure of the sandpile group of a polygon ring depends on the way that the polygons stacked together.

Finally, recall that for any connected graph $G$, the cuts $c_v, v\in V(G)$ generate the whole  cut space of $G$, and deleting any one of them, the remaining cuts are independent and still spanning. While for any plane graph, the cycles corresponding to the  faces generate the whole cycle space of $G$, and deleting any one of them, the remaining cycles are independent and still spanning. In the following, we write $e$ instead of $\delta_e$ for simplicity. Given integers $a_1,\ldots,a_t$, we write $gcd(a_1,\ldots,a_t)$ for their greatest common divisor.

\section{The minimum number of generators of $S(R_n)$ and $S(T_n)$}

In this section, we shall discuss the sandpile group of a polygon ring. Let $n\ge2$ and $a_1,\dots,a_n$, $b_1,\dots,b_n$, $a$, $b$, and $k_1,\ldots,k_n$ be integers as used in the definition of a polygon chain. For the corresponding polygon chain $G_n$, let $e_i=u_iv_i$ $(i=0,1,\ldots,n)$ and the polygons be $C_{k_i}$, $i=1,\dots n$. Recall that $k_i=a_i+b_i+2$ for $1\leq i\leq n$. Let $R_n = R_n(a_1\ldots a_n; b_1\ldots b_n)$ be the polygon ring and $T_n = T_n(a_1\ldots a_n; b_1\ldots b_n)$ be the corresponding twisted ring.

Now we fix an orientation $\mathcal{O}$ of $E(R_n)$ as follows. We orient all edges in $Q^1$ from $v_0$ towards $v_n$ and all edges in $Q^2$ in the opposite way from $u_n$ towards $u_0$, while each edge $e_i=(u_i,v_i)$ is oriented from $u_i$ to $v_i$, $1\leq i\leq n$. See Figure 2(b) for an example.
We use the same orientation for $T_n$, where the identified free edge is oriented from $u_0$ to $v_0$.

\begin{figure}[htbp]
\centering
\scalebox{0.7}{\includegraphics{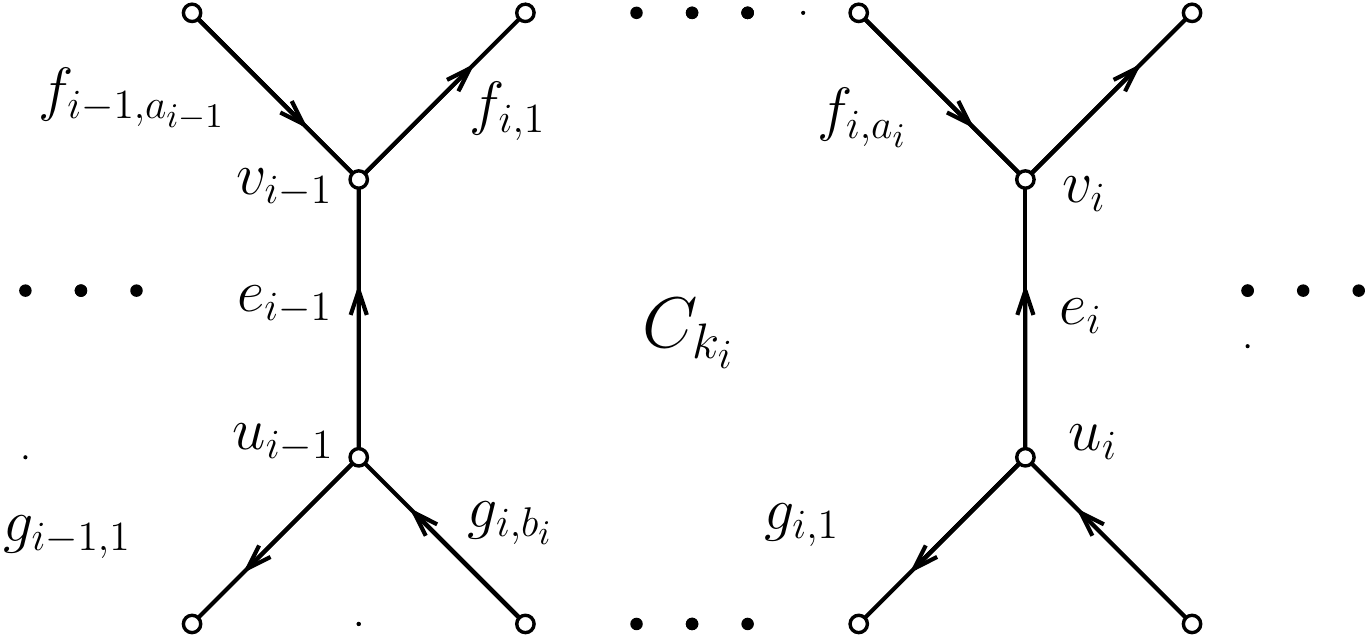}}
\caption{The cycle $C_{k_i}(i<n)$ in an oriented  polygon ring.}
\end{figure}

\begin{thm}\label{thm:3.1}
Let $R_n=R_n(a_1\ldots a_n; b_1\ldots b_n)$ and\/ $T_n=T_n(a_1 \ldots a_n; b_1\ldots b_n)$. The minimum number of generators of their sandpile group $S(R_n)$ and $S(T_n)$ (respectively) is at most $3$.
\end{thm}

\begin{proof} We first prove the result for $R_n$.

It is sufficient to show that there exist at most three edges such that every edge $e\in E(R_n)$ can be expressed by them (viewed as an element in $S(R_n)=\frac{\mathbb{Z}E}{\mathcal{C\oplus \mathcal{B}}}$). We shall use the relations determined by cuts $c_v, v\in V(R_n)$ and polygons $C_{k_i}$ alternatively to prove the result. Let us recall the following simple fact, which we shall use repeatedly without pointing it out every time.
If $e$ and $f$ are two edges incident with a vertex $v$ of degree $2$ and $h(e)=t(f)=v$, then
$$e\equiv f \pmod{c_v}.$$
So bear in mind $f_i=f_{i,1}\equiv \cdots\equiv f_{i,a_i}$ and $g_i=g_{i,1}\equiv \cdots\equiv g_{i,a_i}$.

Also note that, under the given orientation $\mathcal{O}$, the cycle $C_{k_i}$ (see Figure 4) can be expressed as follows:
$$C_{k_i}=e_{i-1}+f_{i,1}+\cdots +f_{i,a_i}-e_i+g_{i,1}+\cdots +g_{i,b_i},\qquad i=1,\dots,n.$$
That is
$$C_{k_i}=e_{i-1}+a_if_i-e_i+b_ig_i,\qquad i=1,\dots,n.\eqno (3.1)$$
In (3.1), if $a_i=0$, then $f_i$ is not defined, and we treat the term $a_if_i$ as being zero. The same convention applies when $b_i=0$.

For clarity, we split the proof in four cases.

(i) There exists some $k$ such that $a_k>0, b_k>0$. Without loss of generality, we may assume that $a_1>1$ and $b_1>1$. We will show that every edge (as an element of $\ZZ E / (\B\oplus\C)$)  can be expressed by the three edges $e_1,f_1,g_1$.

Now we start from the cuts $c_{v_1}, c_{u_1}$ and the cycle $C_{k_2}$:
\begin{align}\label{eq:cuts and cycle at e_i}
  c_{v_1} &= f_1+e_1-f_2 \quad (\textrm{if } a_2>0),\nonumber \\
  c_{u_1} &= g_1+e_1-g_2 \quad (\textrm{if } b_2>0),\nonumber \\
  C_{k_2} &= e_1+a_2f_2+b_2g_2-e_2.\nonumber
\end{align}
From here it is easy to see that $e_2, f_2, g_2$ can be expressed by $e_1,f_1,g_1$. (We use the first and the second equality only when $a_2$ and $b_2$ are nonzero, respectively.) Thus, the same holds for every edge in $E(C_{k_2})$.

Now the proof proceeds by induction on $i$. Suppose that the edges in $E_{i-1}=\cup_{j=2}^{i-1}E(C_{k_j})$ have been expressed by $e_1,f_1,g_1$. If $a_i>0$ and $b_i>0$, then
$$c_{v_{i-1}}=f_{i}-\sum_{e\in E_{i-1},h(e)=v_{i-1}}e\ \qquad \mbox{and}\ \qquad c_{u_{i-1}}=g_i-\sum_{e\in E_{i-1}, t(e)=u_{i-1}}e.\eqno (3.2)$$
By the induction, $f_i$ and $g_i$ can be expressed by $e_1,f_1,g_1$. Then by (3.1), so does $e_i$. Thus, all edges in $C_{k_i}$ can be expressed in this case.  The proof is basically the same when at least one of conditions $a_i=0$ or $b_i=0$ or $a_i=b_i=0$ holds. The details are left to the reader.

(ii)  $a_1=\cdots=a_n=b_1=\cdots=b_n=0$. In this case, by (3.1),
$e_1\equiv \cdots\equiv e_n$. So $\mu(R_n)=1.$

(iii) $a_1=\cdots=a_n=0$ but there exists some $k$ with $b_k>0$, or $b_1=\cdots=b_n=0$ but there exists some $k$ with $a_k>0$. Without loss of generality, suppose $a_1>0$ and $b_1=\cdots=b_n=0$. By a similar argument as in (i), it is easy to show that every edge $e\in E(R_n)$ can be expressed by $e_1$ and $f_1$. So in this case, $\mu(R_n)\leq 2$.

(iv) $a_i\cdot b_i = 0$ for $1\leq i\leq n$, but there exist $k$ and $l$ such that $a_k>0$ and $b_l>0$. Suppose $a_1>0$ and $b_1=\cdots=b_{l-1}=0$, but $b_l>0$. In this case, every edge $e\in E(R_n)$ can be expressed by $e_1$, $f_1$ and $g_l$.

Combining the above together, we obtain the result for $R_n$.

Since the relations determined by cuts $c_{u_{i}}, c_{v_{i}}$ ($i=1,\dots,n-1$) and cycles $C_{k_j}$ ($j=2,\dots, n-1$) in $R_n$ and $T_n$ are the same, it is easy to see that all edges in $T_n$ except $e_n$ can be expressed by $e_1,f_1,g_1$ with the same procedure as in $R_n$. While $e_n$ is determined by the relation
$$e_n=-(e_{n-1}+a_{n}f_n+b_ng_n).$$
So the result holds for $T_n$, too.
\end{proof}

Note that in the above expressing process, we only use the cycles $C_{k_i}$ ($2\leq i\leq n$) and cuts $c_{v_i}, c_{u_i}, i=1,\ldots,n-1$. So there are three independent elements that have not been used: two in the cycle space and one in the cut space. They give a relation matrix for the three generators $e_1,f_1,g_1$.

For $R_n$, we choose:
$$C_{k_1}=a_1f_1+e_n+b_1g_1-e_1,\ C^1 = a_1f_1+a_2f_2+\cdots+a_nf_n \hbox{ and } c_{v_n} = f_n+e_n-f_1.\eqno (3.3)$$

For $T_n$, we choose:
$$C_{k_1}=a_1f_1+e_n+b_1g_1-e_1,\ C^1 = a_1f_1+a_2f_2+\cdots+a_nf_n+e_n \hbox{ and } c_{v_n} = -g_n+e_n-f_1.\eqno (3.4)$$
\smallskip
Recall that the expressions for the edges are the same in $R_n$ and $T_n$ in terms of $e_1,f_1,g_1$, but different for $e_n$, which satisfies $e_n(T_n)=-e_n(R_n)$ in the form. So suppose in $R_n$
 $$\left(
\begin{array}{c}
 f_n\\
 e_n\\
 g_n
  \end{array}
\right)=\left(
\begin{array}{ccc}
 a_{11}&a_{12}&a_{13}\\
 a_{21}&a_{22}&a_{23}\\
 a_{31}&a_{32}&a_{33}
  \end{array}
\right)\left(
\begin{array}{c}
 f_1\\
 e_1\\
 g_1
  \end{array}
\right).$$
Then we have in $T_n$
 $$\left(
\begin{array}{c}
 f_n\\
 e_n\\
 g_n
  \end{array}
\right)=\left(
\begin{array}{ccc}
 a_{11}&a_{12}&a_{13}\\
 -a_{21}&-a_{22}&-a_{23}\\
 a_{31}&a_{32}&a_{33}
  \end{array}
\right)\left(
\begin{array}{c}
 f_1\\
 e_1\\
 g_1
  \end{array}
\right).$$
By expressing 
$$a_1f_1+a_2f_2+\cdots+a_nf_n=c_1f_1+c_2e_1+c_3g_3,$$ 
we obtain the following result.

\begin{thm}\label{thm:3.3}
Let the integers $a_{ij}$ and $c_i$ $(1\le i\le3,\ 1\le j\le3)$ be as introduced above. Then the following holds.

{\rm (1)} For the polygon ring $R_n$, the generators $e_1, f_1$ and $g_1$ of its sandpile group $S(R_n)$ have the
relation matrix:
$$
M=\left(
\begin{array}{ccc}
  a_{11}-a_{1}-1&a_{12}+1&a_{13}-b_1\\
  a_{21}+a_1&a_{22}-1&a_{23}+b_1\\
  c_1&c_2& c_3
 \end{array}
\right).
$$

{\rm (2)} For the twisted polygon ring $T_n$, the generators $e_1, f_1$ and $g_1$ of its sandpile group $S(T_n)$ have the
relation matrix:
$$
T=\left(
\begin{array}{ccc}
  a_{31}+a_{1}+1&a_{32}-1&a_{33}+b_{1}\\
  a_{21}-a_1&a_{22}+1&a_{23}-b_1\\
  c_1-a_{1}&c_2+1& c_3-b_{1}
 \end{array}
\right).
$$
\end{thm}

\begin{proof} Given the notation, it is easy to see the coefficient matrices of  (3.3) and (3.4) are
$$
M_1=\left(
\begin{array}{ccc}
 a_{11}+a_{21}-1&a_{12}+a_{22}&a_{13}+a_{23}\\
  a_{21}+a_1&a_{22}-1&a_{23}+b_1\\
  c_1&c_2& c_3
 \end{array}
\right)\sim
\left(
\begin{array}{ccc}
 a_{11}-a_{1}-1&a_{12}+1&a_{13}-b_1\\
  a_{21}+a_1&a_{22}-1&a_{23}+b_1\\
  c_1&c_2& c_3
 \end{array}
\right)=M
$$
$$
T_1=\left(
\begin{array}{ccc}
  a_{31}+a_{21}+1&a_{32}+a_{22}&a_{33}+a_{23}\\
  a_{21}-a_1&a_{22}+1&a_{23}-b_1\\
  c_1-a_{21}&c_2-a_{22}& c_3-a_{23}
 \end{array}
\right)\sim
\left(
\begin{array}{ccc}
  a_{31}+a_{1}+1&a_{32}-1&a_{33}+b_{1}\\
  a_{21}-a_1&a_{22}+1&a_{23}-b_1\\
  c_1-a_{1}&c_2+1& c_3-b_{1}
 \end{array}
\right)=T.
$$
\end{proof}

Since the number of spanning trees $\tau(R_n)=det(M)$ and $\tau(T_n)=det(T)$, we can get a general relation between $\tau(R_n)$ and $\tau(T_n)$ from the above theorem. In the next section, we shall proceed to give explicit relation  matrices among generators for a family of special (twisted) polygon rings:
$R_n(a\ldots a; b\ldots b)$ and $T_n(a\ldots a; b\ldots b)$, that is, $a_1=\cdots=a_n=a$ and $b_1=\cdots=b_n=b$. In this case we will use a simplified notation $R_n(a,b)$ and $T_n(a,b)$ (see Figure 4 for examples). Note that $R_n(0,0)$ is just the 2-vertex graph with $n$ parallel edges. The graph $R_n(a,0)$ ($a>0$) is the \emph{generalized wheel graph}, and $R_n(1,1)$ and $T_n(1,1)$ are the ladder and the M\"{o}bius ladder graphs, respectively. The sandpile groups of these graphs were discussed in many papers as we mentioned in the introduction.

\begin{figure}[htbp]
\centering
\scalebox{0.7}{\includegraphics{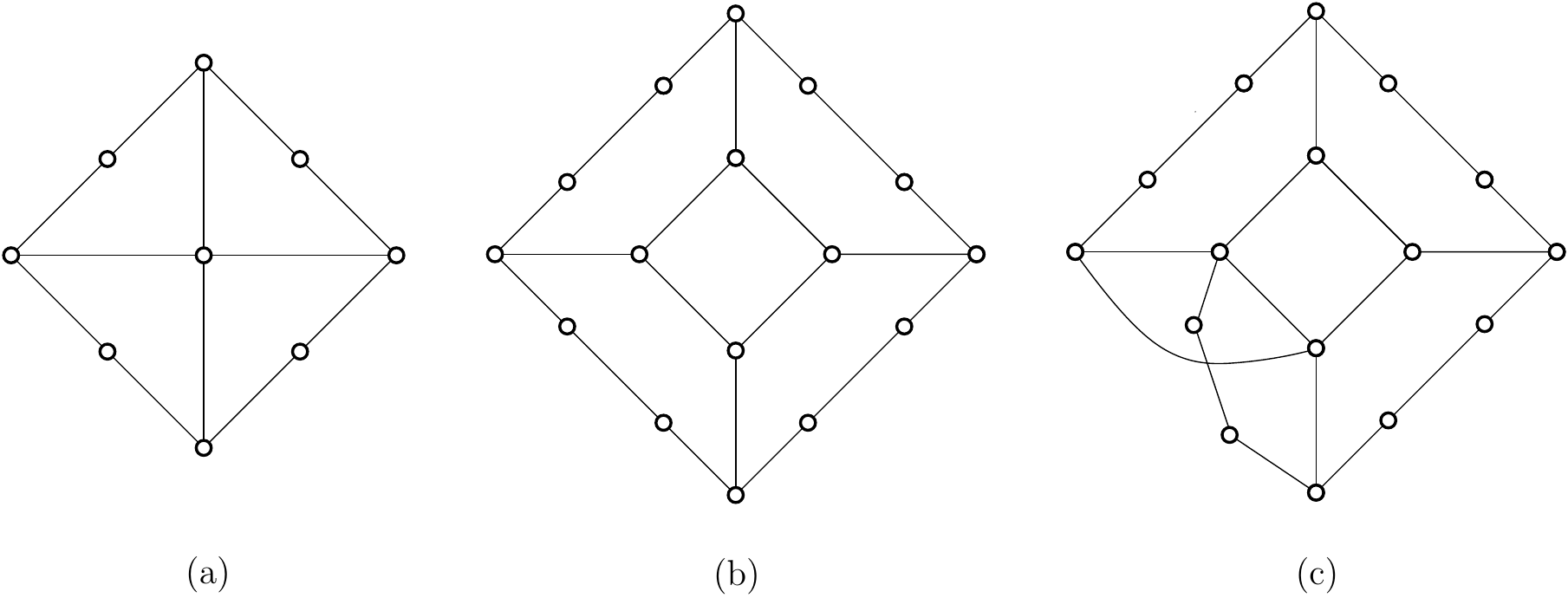}}
\caption{(a) The polygon ring $R_4(2,0)$. (b) The polygon ring $R_4(3,1)$. (c)~The twisted polygon ring $T_4(3,1)$.}
\end{figure}

\section{The relation matrices of $R_n(a,b)$ and $T_n(a,b)$}

In this section, we focus on  the explicit expressions for the relation matrices given in Theorem 3.2 for $R_n(a,b)$ and $T_n(a,b)$.

\begin{lemma}
Let $a\geq b$ be non-negative integers with $a>0$. Let $(\tau_n)_{n\ge1}$ be the integer sequence satisfying $\tau_n=(a+b+2)\tau_{n-1}-\tau_{n-2}$ with initial values $\tau_{-1}=-1$ and $\tau_0=0$. Let $$A=\left(
\begin{array}{ccc}
 1&1&0\\
 a&a+b+1&b\\
 0&1&1
  \end{array}
\right).$$
Then
$$A^n=\left(
\begin{array}{ccc}
  \frac{b}{a+b}+\frac{a}{a+b}(\tau_n-\tau_{n-1})&\tau_n& -\frac{b}{a+b}+\frac{b}{a+b}(\tau_n-\tau_{n-1})\\
  a\tau_n&\tau_{n+1}-\tau_{n}&b\tau_n\\
-\frac{a}{a+b}+\frac{a}{a+b}(\tau_n-\tau_{n-1})&\tau_n& \frac{a}{a+b}+\frac{b}{a+b}(\tau_n-\tau_{n-1})
 \end{array}
\right).$$
\end{lemma}

\begin{proof} We prove the result by induction on $n$. It is trivial to check that the result holds for $n=0,1$. Suppose the result holds for $n=k$. Let us write $\nu = \tau_k-\tau_{k-1}$ and $\nu' = \tau_{k+1}-\tau_k$. Note that $\nu' = \nu + (a+b)\tau_k$. Now it is easy to verify the following:
\begin{align*}
  A^{k+1} &= A^kA \\
   &= 
\left(
\begin{array}{ccc}
  \frac{b}{a+b}+\frac{a}{a+b}\nu+a\tau_k & (a+b+2)\tau_k-\tau_{k-1} & -\frac{b}{a+b}+\frac{b}{a+b}\nu+b\tau_k\\
  a\tau_{k+1} & (a+b+1)\tau_{k+1}-\tau_{k} & b\tau_{k+1}\\
-\frac{a}{a+b}+\frac{a}{a+b}\nu+a\tau_k & (a+b+2)\tau_k-\tau_{k-1 }& \frac{a}{a+b}+\frac{b}{a+b}\nu+b\tau_k
 \end{array}
\right) \\
  &=
\left(
\begin{array}{ccc}
 \frac{b}{a+b}+\frac{a}{a+b}\nu' & \tau_{k+1} & -\frac{b}{a+b}+\frac{b}{a+b}\nu' \\
  a\tau_{k+1} & \tau_{k+2}-\tau_{k+1} & b\tau_{k+1}\\
-\frac{a}{a+b}+\frac{a}{a+b}\nu' & \tau_{k+1} & \frac{a}{a+b}+\frac{b}{a+b}\nu'
 \end{array}
\right).
\end{align*}
This completes the proof. 
\end{proof}

\noindent
{\bf Remark 1.} Let $G_n$ be  the polygon chain $G_n(s,\ldots,s)$, where $s=a+b+2$. Then it is easy to see that $\tau_n$ in Lemma 4.1 is equal to the number of spanning trees of $G_{n-1}$ ($n>0$) since they satisfy the same linear recurrence relation with the same initial values. Solving the linear recurrence relation, we have
$$
  \tau_n = \frac{1}{\sqrt{s^2-4}}\left(\lambda_1^n-\lambda_2^n\right),\eqno (4.1)
$$
where
$$
  \lambda_1 = \frac{s+\sqrt{s^2-4}}{2},\ \ \ \lambda_2 = \frac{s-\sqrt{s^2-4}}{2}
$$
are the roots of $x^2-sx+1=0$.
From (4.1), it is easy to deduce that $$\sum_{i=0}^{n-1}\tau_i=\frac{\tau_n-\tau_{n-1}-1}{a+b}. \eqno (4.2)$$
$$\tau_{n-1}^2-\tau_n\tau_{n-2}=1. \eqno (4.3)$$
\smallskip

In the following, for simplicity, we always assume that $a\geq b$ and $n\geq 2$ for $R_n(a,b)$ and for $T_n(a,b)$.

\begin{thm}\label{thm:4.2}
Let $e_i, f_i$ and $g_i$ be as given in Theorem 3.1. Then the following holds.

{\rm (1)} For the polygon ring $R_n(a,b)$ with  $a\geq b>0$, the generators $e_1, f_1$ and $g_1$ have the
relation matrix:
$$
M=\left(
\begin{array}{ccc}
  \frac{a}{a+b}(\tau_n-\tau_{n-1}-1)&\tau_n&\frac{b}{a+b}(\tau_n-\tau_{n-1}-1)\\
  a(\tau_{n-1}+1)&\tau_{n}-\tau_{n-1}-1&b(\tau_{n-1}+1)\\
  \frac{a}{a+b}(nb+a(\tau_{n-1}+1))&\frac{a}{a+b}(\tau_n-\tau_{n-1}-1)& \frac{b}{a+b}(-na+a(\tau_{n-1}+1))
 \end{array}
\right).
$$

{\rm (2)} For the polygon ring $R_n(a,0)$ with  $a>0$, the generators $e_1, f_1$  have the
relation matrix:
$$
N=\left(
\begin{array}{cc}
 \tau_n-\tau_{n-1}-1&\tau_n\\
  a(\tau_{n-1}+1)& \tau_{n}-\tau_{n-1}-1\\
  \end{array}
\right).
$$

{\rm (3)} For the twisted polygon ring $T_n(a,b)$ with  $a\geq b>0$, the generators $e_1, f_1$ and $g_1$ have the
relation matrix:
$$
T=\left(
\begin{array}{ccc}
  \frac{a}{a+b}(\tau_n-\tau_{n-1}-1)+1&\tau_n&\frac{b}{a+b}(\tau_n-\tau_{n-1}-1)+1\\
  a(\tau_{n-1}-1)&\tau_{n}-\tau_{n-1}+1&b(\tau_{n-1}-1)\\
  \frac{a}{a+b}(nb+a-b\tau_{n-1})&-\frac{b}{a+b}(\tau_n-\tau_{n-1}-1)-1& \frac{b}{a+b}(-na+a-b\tau_{n-1})
 \end{array}
\right).
$$

\end{thm}

\begin{proof} Let us first consider the polygon ring $R_n(a,b)$.
By Theorem 3.1, we know that $e_1, f_1$ and $g_1$ generate the sandpile group $S(R_n)$. Moreover, by (3.1) and (3.2), for $i>1$, we have
\begin{align*}
  f_i & = f_{i-1}+e_{i-1},\\
  g_i & = e_{i-1}+g_{i-1},\\
  e_i & = e_{i-1}+af_i+bg_i=af_{i-1}+(a+b+1)e_{i-1}+bg_{i-1}.
\end{align*}
Let $\alpha_i=(f_i,e_i,g_i)^t$. Then the above relation can be expressed as $\alpha_i=A\alpha_{i-1}$, where $A$ is the matrix in Lemma 4.1.
So $$\alpha_i=A^{i-1}\alpha_1, \ \  i\geq 1.$$
Furthermore, by (4,1) and (4.2), we have
\begin{align*}
  C^1 &= a(f_1+f_2+\cdots+f_n)\\
      &= \frac{a}{a+b}[(nb+a(\tau_{n-1}+1))f_1+(\tau_n-\tau_{n-1}-1)e_1+(-nb+b(\tau_{n-1}+1))g_1].
\end{align*}
Then the matrices $M_1$ and $T_1$ in  Theorem 3.2 have the form $M$ and $T$ in (1) and (3), respectively.

Part (2) follows from (1) by letting $b=0$ and deleting the third relation.
\end{proof}

From the above theorem, we immediately derive the following corollary.

\begin{cor}
Let $R_n=R_n(a,b)$ be a polygon ring, and let $T_n=T_n(a,b)$ be a twisted polygon ring. The number of spanning trees of these graphs is given by the formula:
$$
  \tau(R_n) = |S(R_n)| = 
\begin{cases}
\frac{nab}{a+b}(\tau_{n+1}-\tau_{n-1}-2)=\frac{nab}{a+b}(\lambda_1^n+\lambda_2^n-2), &\mbox{if } a>0,\ b>0\\
\tau_{n+1}-\tau_{n-1}-2=\lambda_1^n+\lambda_2^n-2,  &\mbox{if } a>0,\ b=0
\end{cases}
$$
and
$$
  \tau(T_n) = |S(T_n)| =
\frac{nab}{a+b}(\tau_{n+1}-\tau_{n-1}-2)+\frac{4nab}{a+b}+\frac{(a-b)^2}{a+b}\tau_n.
$$
\end{cor}

\begin{proof} We only give the proof for $R_n(a,b)$ with $a>0$, $b>0$. The other cases can be deduced similarly. By Theorem 4.2, $M$ is a relation matrix of $S(R_n)$ for $a>0$, $b>0$. So
\begin{align*}
 \tau(R_n) &= |S(R_n)| = |det(M)| \\
 &= \left|
\begin{array}{ccc}
  \frac{a}{a+b}(\tau_n-\tau_{n-1}-1)&\tau_n&\frac{b}{a+b}(\tau_n-\tau_{n-1}-1)\\
  a(\tau_{n-}+1)&\tau_{n}-\tau_{n-1}-1&b(\tau_{n-1}+1)\\
  \frac{a}{a+b}(nb+a(\tau_{n-1}+1))&\frac{a}{a+b}(\tau_n-\tau_{n-1}-1)& \frac{b}{a+b}(-na+a(\tau_{n-1}+1))
 \end{array}
   \right| \\
 &= \left|
\begin{array}{ccc}
  \frac{a}{a+b}(\tau_n-\tau_{n-1}-1)&\tau_n&0\\
  a(\tau_{n-1}+1)&\tau_{n}-\tau_{n-1}-1&0\\
  \frac{a}{a+b}(nb+a(\tau_{n-1}+1))&\frac{a}{a+b}(\tau_n-\tau_{n-1}-1)& -nb
 \end{array}
    \right| \\
 &=-nab\left(\frac{1}{a+b}(
\tau_n-\tau_{n-1}-1)^2-\tau_n(\tau_{n-1}+1)\right)\\
&=\frac{nab}{a+b}(\tau_{n+1}-\tau_{n-1}-2)\\
&=\frac{nab}{a+b}(\lambda_1^n+\lambda_2^n-2)
\end{align*}
where the next to the last equality is obtained from the recurrence relation and (4.3).
\end{proof}

\noindent
{\bf Remark 2.} Recall that $\tau_n$ is equal to the number of spanning trees of polygon chain $G_{n-1}=G_{n-1}(a+b+2,\ldots, a+b+2)$. As a by-product, we have the following elegant relations:
$$
\tau(R_n(a,b) = \left\{
\begin{aligned}
  &\frac{nab}{a+b}(\tau(G_{n})-\tau(G_{n-2})-2), \ \ \mbox {if}\ \ a>0,\ b>0;\\
 &\tau(G_{n})-\tau(G_{n-2})-2, \ \ \mbox {if}\ \  a>0,\ b=0;
 \end{aligned}
\right.$$
$$\tau(T_n(a,b)) = \tau(R_n(a,b))+\frac{4nab}{a+b}+\frac{(a-b)^2}{a+b}\tau_n.$$
Specially, $$
\tau(T_n(a,a)) = \tau(R_n(a,a))+2na.$$

In the next two sections, we shall determine the sandpile group of $R_n(a,b)$ and $T_n(a,b)$ by computing  the Smith Normal Forms of the matrices $M$ and $T$.

\section{The sandpile group of $R_n(a,b)$}

In this section, we shall focus on the Smith Normal Form of $M$ in Theorem 4.2. First we give some divisibility properties of $\tau_n$. Recall that $\tau_n$ is an integer sequence satisfying $\tau_n=(a+b+2)\tau_{n-1}-\tau_{n-2}$ with initial values $\tau_{-1}=-1$ and $\tau_0=0$. For clarity, we define three new integer sequences $\beta_n, \gamma_n, \delta_n$ that satisfy the same linear recurrence relation as $\tau_n$ with initial values
$$
\aligned
& \beta_0=2,\ \beta_1=a+b+2;\\
& \gamma_0=-1,\ \gamma_1=1;\\
& \delta_0=1,\ \delta_1=a+b+1.
\endaligned
$$

Let $s=a+b+2$ and $t=\sqrt{s^2-4}$. Then the two roots of $x^2-sx+1$ are
$$\lambda_1=\frac{s+t}{2}, \ \ \lambda_2=\frac{s-t}{2}$$
and
$$
\aligned
& \tau_n=\frac{1}{t}(\lambda_1^n-\lambda_2^n),\\
& \beta_n=\lambda_1^n+\lambda_2^n,\\
& \gamma_n=\frac{-(a+b)+t}{2(a+b)}\lambda_1^n-\frac{(a+b)+t}{2(a+b)}\lambda_2^n,\\ &\delta_n=\frac{a+b+4+t}{2(a+b+4)}\lambda_1^n+\frac{a+b+4-t}{2(a+b+4)}\lambda_2^n.
\endaligned 
$$
Furthermore, we write
$$ \rho_n=\frac{\tau_n-\tau_{n-1}-1}{a+b},\ \ s_n=\frac{\tau_{n+1}-\tau_{n-1}-2}{a+b}.$$
Then we have the following results.

\begin{lemma}
\rm{(1)} If $n$ is even, then
$$
\begin{aligned}
&\tau_n=\tau_{\frac{n}{2}}\beta_{\frac{n}{2}}; \\
&\tau_{n-1}+1=\tau_{\frac{n}{2}}\beta_{\frac{n}{2}-1};\\
&\tau_{n-1}-1=\tau_{\frac{n}{2}-1}\beta_{\frac{n}{2}};\\
&\rho_n=\tau_{\frac{n}{2}}\gamma_{\frac{n}{2}};\\
&s_n=(a+b+4)\tau_{\frac{n}{2}}^2.
\end{aligned}
$$

\rm{(2)} If $n$ is odd, then
$$
\begin{aligned}
&\tau_n=\gamma_{\frac{n+1}{2}}\delta_{\frac{n-1}{2}};\\
&\tau_{n-1}+1=\gamma_{\frac{n+1}{2}}\delta_{\frac{n-3}{2}};\\
&\tau_{n-1}-1=\gamma_{\frac{n-1}{2}}\delta_{\frac{n-1}{2}};\\
&\rho_n=\tau_{\frac{n-1}{2}}\gamma_{\frac{n+1}{2}};\\
&s_n=\gamma_{\frac{n+1}{2}}^2.
\end{aligned}
$$
\end{lemma}

\begin{proof}
Using the facts $\lambda_1\lambda_2 = 1$, $\lambda_1+\lambda_2=s$, $\lambda_1-\lambda_2=t$, the identities are easy to check.
\end{proof}

\begin{lemma}
For $\tau_n, \beta_n, \gamma_n, \delta_n$ as defined above, we have
$$gcd(\tau_n,\tau_{n-1})=gcd(\gamma_n,\gamma_{n-1})=gcd(\delta_n,\delta_{n-1})=1$$
and 
$$
gcd(\beta_n,\beta_{n-1}) = 
\begin{cases} 1, &\mbox{if } a+b \mbox{ is odd},\\  2, &\mbox{if } a+b \mbox{ is even}.\end{cases}
$$
\end{lemma}

\begin{proof} 
We only need to notice that, for any sequence $x_n$ satisfying $x_n=(a+b+2)x_{n-1}-x_{n-2}$,
$$gcd(x_n, x_{n-1})=gcd(x_{n-1},x_{n-2})=\cdots=gcd(x_1,x_0).$$
Then the results follow directly.
\end{proof}

\begin{thm}\label{thm:5.3} Let $a>0$ be an integer. The sandpile group of $R_n(a,0)$ is equal to
$$
S(R_n(a,0))=\left\{
\begin{aligned}
\mathbb{Z}_{\tau_{n}}\oplus \mathbb{Z}_{a(a+4)\tau_{n}},\ \ \  &\mbox{if $n$ is even and $a$ is odd};\\
\mathbb{Z}_{2\tau_{n}}\oplus \mathbb{Z}_{\frac{a(a+4)}{2}}\tau_{n},\ \ \  &\mbox{if $n$ is even and $a$ is even};\\
\mathbb{Z}_{\gamma_{m+1}}\oplus \mathbb{Z}_{a\gamma_{m+1}},\ \ \  &\mbox{if $n=2m+1$ is  odd}.
\end{aligned}
\right.
$$
\end{thm}

\begin{proof} By Theorem 4.2, the relation matrix of $R_n(a,0)$ is  
$$N=\left(
\begin{array}{cc}
 \tau_n-\tau_{n-1}-1&\tau_n\\
  a(\tau_{n-1}+1)& \tau_{n}-\tau_{n-1}-1\\
  \end{array}
\right)=\left(
\begin{array}{cc}
 \tau_n-\tau_{n-1}-1&\tau_n\\
  a\tau_n& (a+1)\tau_n-\tau_{n-1}-1\\
  \end{array}
\right).$$
So the determinant factors of $N$ are $\Delta_1=gcd(\tau_n, \tau_{n-1}+1)$ and $\Delta_2=\tau_{n+1}-\tau_{n-1}-2$. Then the result follows directly from Theorem 2.3 and Lemmas 5.1--5.2.
\end{proof}

\noindent
{\bf Remark 3.} Note that $R_n(1,0)$ is just the wheel graph $W_n$, whose sandpile group has been determined before, see  \cite{Biggs1999}:
$$S(W_n)=\left\{
\begin{aligned}
&\mathbb{Z}_{\eta_{n}}\oplus \mathbb{Z}_{5\eta_n},\ \ \ \mbox{if $n$ is even;}\\
&\mathbb{Z}_{\eta_{n+1}+\eta_{n-1}}\oplus \mathbb{Z}_{\eta_{n+1}+\eta_{n-1}},\ \ \ \mbox{if $n$ is odd;}\\
\end{aligned}
\right.
$$
where $\eta_n=\frac{1}{\sqrt{5}}((\frac{1+\sqrt{5}}{2})^n-(\frac{1-\sqrt{5}}{2})^n)$ are Fibonacci numbers. In \cite{Raza2015On}, the authors considered the sandpile group of subdivided wheel graphs isomorphic to $R_n(2,0)$. It is easy to check that our results are consistent with the known results in both cases $a=1$ and $a=2$. Here, we obtain the general result by using a completely different method. Theorem 5.3 also tells us that the minimum number of generators of $S(R_n(a,0))$ is $2$ if $n\geq 3$.

Now for the sandpile groups of $R_n(a,b)$  with $a\geq b>0$.   First we simplify the relation matrix $M$  in Theorem 4.2 as follows: first  adding $a+b$ times and $a$ times the first row to the second and the last row, respectively, we have
\begin{align*}
M &\sim \left(
\begin{array}{ccc}
  a\cdot\frac{\tau_n-\tau_{n-1}-1}{a+b}&\tau_n&b\cdot\frac{\tau_n-\tau_{n-1}-1}{a+b}\\
  a\tau_{n}&\tau_{n+1}-\tau_{n}-1&b\tau_{n}\\
  a\tau_n-ab\cdot\frac{\tau_n-n}{a+b}&a\cdot\frac{\tau_{n+1}-\tau_{n}-1}{a+b}& ab\cdot\frac{\tau_n-n}{a+b}
 \end{array}
\right) \\
 &=\left(
\begin{array}{ccc}
  a\rho_n&\tau_n& b\rho_n\\
  a\tau_{n}&(a+b)\rho_{n+1}&b\tau_{n}\\
  a\tau_n-ab\tau'_n&a\rho_{n+1}& ab\tau'_{n}
 \end{array}
\right)=M',
\end{align*}
where  $\tau'_n=\frac{\tau_n-n}{a+b}=\rho_n+\rho_{n-1}+\cdots+\rho_1$ are integers.

In order to obtain the determinant factors of $M'$, we compute all order-$2$ minors of $M'$. Let $M'_{ij;kl}$ denote the minor corresponding to rows $i,j$ and columns $k,l$. First, by using (3.5) and $ \tau_{n+1}=(a+b+2)\tau_n-\tau_{n-1},$ it is easy to check that
$$\rho_{n+1}=\tau_n+\rho_n;\ \ \tau_n^2-(a+b)\rho_n\rho_{n+1}=s_n. \eqno (5.1).$$
Bearing in mind (5.1) and $\tau_n-(a+b)\tau'_n=n$, we have
$$\aligned
&M'_{12;12}=-a s_n;\ \ M'_{12;13}=0; \ \ M'_{12;23}=b s_n;\\
&M'_{13;12}=-a\cdot \frac{as_n+nb\tau_n}{a+b}=-as_n+ab\cdot\frac{s_n-n\tau_n}{a+b};\\
& M'_{13;13}=-nab\rho_n;\ \
M'_{13;23}=ab\cdot\frac{s_n-n\tau_n}{a+b};\\ &M'_{23;12}=-nab\rho_{n+1}=-nab(\tau_n+\rho_n);\ \
M'_{23;13}=-nab\tau_n;\ \ M'_{23;23}=-nab\rho_{n+1}.
\endaligned
$$

So it is not difficult to deduce that
 $$\aligned &\Delta_1(M')=gcd(\tau_n, a\rho_n, \ a\rho_n, \ ab\tau'_n) = gcd(\tau_n, \ \tau_{n-1}+1, \ a\rho_n, \ ab\tau'_n);\\
 &\Delta_2(M')=gcd (nab\tau_n, \ nab\rho_n,\ a s_n, \ b s_n,\ ab\cdot\frac{s_n-n\tau_n}{a+b});\\
 &\Delta_3(M')=nab\cdot s_n.
 \endaligned
 $$

\begin{lemma}
Suppose that $\tau_n, \beta_n, \gamma_n$ and $\delta_n$ are as above. Let
$$ \tau'_n=\frac{\tau_n-n}{a+b},\ \ \beta'_n=\frac{\beta_n-2}{a+b},\ \ \gamma'_n=\frac{\gamma_n-2n+1}{a+b},\textrm{ and } \delta'_n=\frac{\delta_n-1}{a+b}.$$
Then $\tau'_n, \beta'_n, \gamma'_n$ and $\delta'_n$ are integers.
\end{lemma}

\begin{proof} 
Note that $\tau_n, \beta_n, \gamma_n$ and $\delta_n$ all satisfy the recurrence relation $x_n=(a+b+2)x_{n-1}-x_{n-2}$. So they can be expressed as the polynomials of $a+b$. It is easy to show that their constant terms are $n, 2, 2n-1$ and $1$, respectively. So the results follow directly.
\end{proof}

\begin{thm}\label{thm:5.5}
Let $a\geq b>0$ be integers. Then the sandpile group of $R_n(a,b)$ is equal to
$$
S(R_n(a,b))=\mathbb{Z}_{\Delta_1}\oplus \mathbb{Z}_{\frac{\Delta_2}{\Delta_1}}\oplus \mathbb{Z}_{\frac{\Delta_3}{\Delta_2}},
$$
where
$$
\Delta_1=
\begin{cases}
 gcd(gcd(2,a,b) \tau_{m}, \ 2ab\tau'_m), &\mbox{if } n=2m \mbox{ is even,} \\
 gcd(\gamma_{m+1}, \ ab\gamma'_{m+1}), &\mbox{if } n=2m+1 \mbox{ is odd};
\end{cases}
$$
$$
\Delta_2 =
\begin{cases}
 \tau_{m}\cdot gcd(gcd(a,b)(a+b+4)\tau_{m}, ab\cdot gcd(2m,(a+b+4)\tau'_m+m)), &\mbox{if } n=2m \hbox{ is even;} \\
 \gamma_{m+1}\cdot gcd ( gcd(a,b)\gamma_{m+1}, ab\cdot gcd(2m+1,\gamma'_{m+1})), &\mbox{if } n=2m+1 \hbox{ is odd};
\end{cases}
$$
and
$$\Delta_3 =
\begin{cases}
 nab(a+b+4)\tau_{m}^2, &\mbox{if } n=2m \hbox{ is even;} \\ 
 nab\gamma_{m+1}^2, &\mbox{if } n=2m+1 \hbox{ is odd}.
\end{cases}
$$
\end{thm}

\begin{proof} We only need to show that the determinant factors of $M'$ have the claimed form in the theorem.

(i) If $n=2m$ is even, then by Lemma 5.1, we have
$$
\aligned &\tau_n=\tau_{m}\beta_{m};\ \  \tau_{n-1}+1=\tau_{m}\beta_{m-1},\\
&\rho_n=\tau_{m}\gamma_{m};\ \  s_n=(a+b+4)\tau_{m}^2.
\endaligned
$$
Recall from Lemma 5.2 that
$gcd(\beta_n, \beta_{n-1})=2$ if $a+b+2$ is even, and $gcd (\beta_n, \beta_{n-1})=1$, otherwise.  And $\gamma_n$ are all odd. Combining the above facts together, we deduce that
\begin{align*}
\Delta_1(M')&=gcd (\tau_n, \ \tau_{n-1}+1, \ a\rho_n, \ ab\tau'_n)\\
&=\begin{cases}
  gcd (gcd(2,a) \tau_{m}, \ ab\tau'_n), &\mbox{if } a+b \mbox{ is even}; \\
  gcd (\tau_{m}, \ ab\tau'_n), &\mbox{if } a+b \mbox{ is odd}.
\end{cases} \\
&=gcd (gcd(2,a,b) \tau_{m}, \ ab\tau'_n)\\
&=gcd(gcd(2,a,b) \tau_{m}, \ 2ab\tau'_m).
\end{align*}
The last equality was obtained by using Lemma 5.3, since
$$\tau'_n=\frac{\tau_{2m}-2m}{a+b}=\frac{\tau_{m}\beta_m-2m}{a+b}=\frac{\tau_{m}((a+b)\beta'_m+2)-2m}{a+b}=\tau_m\beta'_m+2\tau'_m.$$
From the fact that $gcd(\tau_n, (a+b)\rho_n)=gcd(\tau_n, \tau_{n-1}+1)=\tau_{m}\cdot gcd(2,a+b)$, we deduce that $gcd(\tau_n, \rho_n) = \tau_{m}\cdot gcd(\beta_{m},\gamma_{m}) = \tau_m$. So,
$$\Delta_2(M')=\tau_{m}\cdot gcd (nab, \ gcd(a,b)(a+b+4)\tau_{m},~ab\cdot \frac{(a+b+4)\tau_m-n\beta_m}{a+b}).$$
Using Lemma 5.3 again, we have
$$\frac{(a+b+4)\tau_m-n\beta_m}{a+b}=\frac{(a+b+4)((a+b)\tau'_m+m)-n((a+b)\beta'_m+2)}{a+b}=(a+b+4)\tau'_m+m-n\beta'_m.$$
Hence, 
\begin{align*}
\Delta_2(M')&=\tau_{m}\cdot gcd (\ gcd(a,b)(a+b+4)\tau_{m},\, ab\cdot gcd(2m,(a+b+4)\tau'_m+m)),\\
\Delta_3(M')&=nab(a+b+4)\tau_{m}^2. 
\end{align*}

(ii) If $n=2m+1$ is odd, then by Lemma 5.1 again, we have
$$\aligned &\tau_n=\gamma_{m+1}\delta_{m},\ \  \tau_{n-1}+1=\gamma_{m+1}\delta_{m-1},\\
&\rho_n=\gamma_{m+1}\tau_{m},\ \  s_n=\gamma_{m+1}^2.\\
\endaligned
$$
In this case, first note that
$$gcd(\tau_n, (a+b)\rho_n)=gcd(\tau_n, \tau_{n-1}+1)=\gamma_{m+1}=gcd(\tau_n,\rho_n)$$ and
$$\tau'_n=\frac{\gamma_{m+1}\delta_m-2m-1}{a+b}=\frac{\gamma_{m+1}((a+b)\delta'_m+1)-2m-1}{a+b}=\gamma_{m+1}\delta'_m+\gamma'_{m+1}.$$
So
\begin{align*}
\Delta_1(M')&=gcd (\gamma_{m+1}, \ ab\tau'_n)=gcd (\ \gamma_{m+1}, \ ab\gamma'_{m+1});\\
 \Delta_2(M')&=\gamma_{m+1}\cdot gcd (nab, \ gcd(a,b)\cdot \gamma_{m+1}, \ ab\cdot\frac{\gamma_{m+1}-n\delta_m}{a+b})\\
 &=\gamma_{m+1}\cdot gcd(nab,\ gcd(a,b)\cdot \gamma_{m+1}, ab\cdot\frac{(a+b)\gamma'_{m+1}+2m+1-n((a+b)\delta'_m+1)}{a+b})\\
 &=\gamma_{m+1}\cdot gcd(\ gcd(a,b)\cdot \gamma_{m+1}, ab\cdot gcd(n,\gamma'_{m+1}));\\
 \Delta_3(M')&=nab\gamma_{m+1}^2.
\end{align*}
This completes the proof.
\end{proof}

Note that Theorem 5.3 can be viewed as a special case of Theorem 5.5 with $b=0$. Considering two cases, $a=b$ and $a>b=1$, we obtain the following corollaries.
\begin{cor}
Let $a>0$ be an integer. Then for the sandpile group of $R_n(a,a)$, we have
$$S(R_n(a,a))=\mathbb{Z}_{\Delta_1}\oplus \mathbb{Z}_{\frac{\Delta_2}{\Delta_1}}\oplus \mathbb{Z}_{\frac{\Delta_3}{\Delta_2}},$$
where
\begin{align*}
\Delta_1&=
\begin{cases}
gcd(gcd(2,a)\tau_{m}, \ am), &\mbox{if } n=2m; \\
gcd(\gamma_{m+1}, \ a(2m+1)), &\mbox{if } n=2m+1.
\end{cases}
\\
\Delta_2&=
\begin{cases}
a\tau_{m}\cdot gcd(4m, \ 2ma, (a+2)\tau_m-2m), &\mbox{if } n=2m; \\
a\gamma_{m+1}\cdot gcd(2m+1, \gamma_{m+1}), &\mbox{if } n=2m+1 \mbox{ is odd;}
\end{cases}
\\
\Delta_3&=
\begin{cases}
4ma^2(a+2)\tau_{m}^2, &\mbox{if } n=2m; \\
(2m+1)a^2\gamma_{m+1}^2, &\mbox{if } n=2m+1.
\end{cases}
\end{align*}
\end{cor}

\begin{proof}
(i) Note that if $a=b$, then $\tau_m=2a\tau'_m+m$  and $\gamma_{m+1}=2a\gamma'_{m+1}+2m+1$, that is $gcd(\gamma_{m+1},a\gamma'_{m+1})=gcd(\gamma_{m+1}, 2m+1).$ So
$$gcd (\ gcd(2,a)\tau_{m}, \ 2a^2\tau'_m)=gcd (\ gcd(2,a)\tau_{m}, \ a(\tau_m-m))=gcd (\ gcd(2,a)\tau_{m}, am)$$
and $$gcd(\gamma_{m+1}, a^2\gamma'_{m+1})=gcd(\gamma_{m+1},a(2m+1)).$$
Similarly, we can simplify $\Delta_2$ and $\Delta_3$ to obtain the claimed expressions.
\end{proof}

\begin{cor}
Let $a>1$ be an integer. Then for the sandpile group of $R_n(a,1)$, we have
$$S(R_n(a,1))=\mathbb{Z}_{\Delta_1}\oplus \mathbb{Z}_{\frac{\Delta_2}{\Delta_1}}\oplus \mathbb{Z}_{\frac{\Delta_3}{\Delta_2}},$$
where 
\begin{align*}
\Delta_1&=
\begin{cases}
gcd (\tau_{m}, \ 2a\tau'_m), &\mbox{if } n=2m, \\
gcd (\gamma_{m+1}, \ a\gamma'_{m+1}), &\mbox{if } n=2m+1;
\end{cases}
\\
\Delta_2&=
\begin{cases}
\tau_{m}\cdot gcd(\tau_m+4m, \ 2ma), &\mbox{if } n=2m, \\
\gamma_{m+1}\cdot gcd(\gamma_{m+1}, a\gamma'_{m+1}), &\mbox{if } n=2m+1 \mbox{ is odd;}
\end{cases}
\\
\Delta_3&=
\begin{cases}
2ma(a+5)\tau_{m}^2, &\mbox{if } n=2m; \\
(2m+1)a\gamma_{m+1}^2, &\mbox{if } n=2m+1.
\end{cases}
\end{align*}
\end{cor}

\noindent
{\bf Remark 4.} Note that the dual of a polygon ring is just a generalized suspension of a cycle. So the results obtained above are all applicable to the generalized suspension of cycles since the sandpile groups of a connected planar graph and its dual are isomorphic \cite{Cori2000On}.

\section{The sandpile group of  $T_n(a,b)$ }

Similarly, for the relation matrix $T$  in Theorem 4.2, by first  adding $a+b$ times and $a$ times the first row to the second and the last row, respectively, we have

\begin{align*}
T &\sim\left(
  \begin{array}{ccc}
  a\rho_n+1&\tau_n& b\rho_n+1\\
  a\tau_{n}-(a-b)&(a+b)\rho_{n+1}+2&b\tau_{n}+a-b\\
  a\tau_n-ab\tau'_n-a\tau_{n-1}+a&a\rho_{n+1}-\tau_n+\tau_{n-1}& ab\tau'_{n}-b\tau_{n-1}+a
  \end{array}
  \right) \\
  &\sim\left(
  \begin{array}{ccc}
  a\rho_n+1&\tau_n& b\rho_n+1\\
  a\tau_{n}-(a-b)&(a+b)\rho_{n+1}+2&b\tau_{n}+a-b\\
  a-b-ab\tau'_n&-b\rho_{n+1}-1& ab\tau'_{n}-b\tau_{n}
  \end{array}
  \right) \\
  &\sim\left(
  \begin{array}{ccc}
  a\rho_n+1&\tau_n& b\rho_n+1\\
  a\tau_{n}-(a-b)&(a+b)\rho_{n+1}+2&b\tau_{n}+a-b\\
  a\tau_n-ab\tau'_n&a\rho_{n+1}+1& ab\tau'_{n}+a-b
 \end{array}
  \right)=T'.
\end{align*}

It is easy to see that
$$\Delta_1(T')=gcd(\tau_n,\  a\rho_n+1,\  ab\tau'_n,\ a-b);
\ \ \Delta_3(T')=nabs_n+\frac{4nab}{a+b}+\frac{(a-b)^2}{a+b}\tau_n.$$
But the expression for $\Delta_2(T')$ is cumbersome. So we only give the result for $a=b$. In this case
$$\aligned
&T'_{12;12}=-T'_{12;23}=as_n+2; \ \ T'_{12;13}=0; \\
&T'_{13;12}=a\cdot \frac{s_n-n\tau_n}{2}+1;\ \
 T'_{13;13}=-na(a\rho_n+1);\\
&T'_{13;23}=a\cdot\frac{s_n-n\tau_n}{2}-a s_n-1;\\ &T'_{23;12}=T'_{23;23}=-na(a\rho_{n+1}+1);\ \
T'_{23;13}=-na^2\tau_n.
\endaligned
$$
Hence
$$\begin{aligned}
\Delta_1&=gcd(\tau_n,\  a\rho_n+1, \ a^2\tau'_n);\\
\Delta_2&=gcd(na^2\tau_n,\ a s_n+2,\  na(a\rho_n+1),  \frac{as_n+2-na\tau_n}{2});\\
\Delta_3&=na(a s_n+2).
\end{aligned}$$
\begin{thm}
Let $a>0$ be an integer. Then for the sandpile group of $T_n(a,a)$, we have
$$S(T_n(a,a))=\mathbb{Z}_{\Delta_1}\oplus \mathbb{Z}_{\frac{\Delta_2}{\Delta_1}}\oplus \mathbb{Z}_{\frac{\Delta_3}{\Delta_2}},$$
where
$$\Delta_1=\left\{
\begin{aligned}
&gcd (\ \beta_{m}/2, \ am);\ \ \mbox{if  $n=2m$;} \\
&gcd (\ \delta_{m}, \ a(2m+1)),\ \ \mbox{if $n=2m+1$}.
\end{aligned}
\right.
$$
$$
\Delta_2=\left\{
\begin{aligned}
&\beta_{m}/2\cdot gcd (na, \ \beta_m,  \beta_{m}/2-na\tau_m);\ \ \mbox{if  $n=2m$;} \\
&\delta_m\cdot gcd (na, \frac{\beta_m+\beta_{m+1}}{2},\frac{\beta_m+\beta_{m+1}-2na\gamma_{m+1}}{4}),\ \  \mbox{if $n=2m+1$ is  odd;}
\end{aligned}
\right.
$$
and
$$\Delta_3=\left\{
\begin{aligned}
&ma\beta_{m}^2;\ \ \mbox{if  $n=2m$;} \\
&(2m+1)a\cdot \frac{\delta_m(\beta_m+\beta_{m+1})}{2},\ \  \mbox{if $n=2m+1$.}
\end{aligned}
\right.
$$
\end{thm}

\noindent
{\bf Remark 5.} Note that the polygon ring is obtained by identifying two free edges of a polygon chain. It is known that any free edge can generate the sandpile group of the polygon chain. So we wonder whether the result we obtained above may hold in a more general setting. That is, let $H$ be a plane graph with cyclic sandpile group, and $e$ and $f$ be two generating edges on the outer face. Let $G$ be the graph obtained from $H$ by identifying $e$ and $f$. Does it follow that the sandpile group of $G$ is the direct sum of at most three cyclic groups?

\bibliographystyle{abbrv}

\bibliography{polygonflower1}

\begin{thebibliography}{10}

\bibitem{Alfaro2012On}
C.~A. Alfaro and C.~E. Valencia.
\newblock On the sandpile group of the cone of a graph.
\newblock {\em Linear Algebra Appl.}, 436(5):1154--1176, 2012.

\bibitem{Bacher1997The}
R.~Bacher, P.~de~la Harpe, and T.~Nagnibeda.
\newblock The lattice of integral flows and the lattice of integral cuts on a
  finite graph.
\newblock {\em Bull. Soc. Math. France}, 125(2):167--198, 1997.

\bibitem{Bai2003On}
H.~Bai.
\newblock On the critical group of the {$n$}-cube.
\newblock {\em Linear Algebra Appl.}, 369:251--261, 2003.

\bibitem{Bak1987Self}
P.~Bak, C.~Tang, and K.~Wiesenfeld.
\newblock Self-organized criticality: An explanation of the 1/f noise.
\newblock {\em Physical Review Letters}, 59(4):381, 1987.

\bibitem{Bak2007Riemann}
M.~Baker and S.~Norine.
\newblock Riemann-{R}och and {A}bel-{J}acobi theory on a finite graph.
\newblock {\em Adv. Math.}, 215(2):766--788, 2007.

\bibitem{Becker2016Cyclic}
R.~Becker and D.~B. Glass.
\newblock Cyclic critical groups of graphs.
\newblock {\em Australas. J. Combin.}, 64:366--375, 2016.

\bibitem{Berget2012The}
A.~Berget, A.~Manion, M.~Maxwell, A.~Potechin, and V.~Reiner.
\newblock The critical group of a line graph.
\newblock {\em Annals of Combinatorics}, 16(3):449--488, 2012.

\bibitem{Biggs1997Algebraic}
N.~L. Biggs.
\newblock Algebraic potential theory on graphs.
\newblock {\em Bull. London Math. Soc.}, 29(6):641--682, 1997.

\bibitem{Biggs1999}
N.~L. Biggs.
\newblock Chip-firing and the critical group of a graph.
\newblock {\em J. Algebraic Combin.}, 9(1):25--45, 1999.

\bibitem{Bond2016The}
B.~Bond and L.~Levine.
\newblock Abelian networks {III}: {T}he critical group.
\newblock {\em J. Algebraic Combin.}, 43(3):635--663, 2016.

\bibitem{Chan2014Sandpile}
S.~H. Chan and D.~Hollmann, H.D.L.and~Pasechnik.
\newblock Sandpile groups of generalized de bruijn and kautz graphs and
  circulant matrices over finite fields.
\newblock {\em Journal of Algebra}, 421:268--295, 2014.

\bibitem{Chandler2015The}
D.~B. Chandler, P.~Sin, and Q.~Xiang.
\newblock The {S}mith and critical groups of {P}aley graphs.
\newblock {\em J. Algebraic Combin.}, 41(4):1013--1022, 2015.

\bibitem{CHEN201968}
H.~Chen and B.~Mohar.
\newblock The sandpile group of a polygon flower.
\newblock {\em Discrete Applied Mathematics}, 270:68 -- 82, 2019.

\bibitem{Chen2008On}
P.~Chen and Y.~Hou.
\newblock On the sandpile group of {$P_4\times C_n$}.
\newblock {\em European J. Combin.}, 29(2):532--534, 2008.

\bibitem{Chen2006On}
P.~Chen, Y.~Hou, and C.~Woo.
\newblock On the critical group of the {M}\"obius ladder graph.
\newblock {\em Australas. J. Combin.}, 36:133--142, 2006.

\bibitem{Christianson2002The}
H.~Christianson and V.~Reiner.
\newblock The critical group of a threshold graph.
\newblock {\em Linear Algebra Appl.}, 349:233--244, 2002.

\bibitem{Cori2000On}
R.~Cori and D.~Rossin.
\newblock On the sandpile group of dual graphs.
\newblock {\em European J. Combin.}, 21(4):447--459, 2000.

\bibitem{Sandpile2003}
A.~Dartois, F.~Fiorenzi, and P.~Francini.
\newblock Sandpile group on the graph {$D_n$} of the dihedral group.
\newblock {\em Euopean J. Comb}, 24:815--824, 2003.

\bibitem{Deryagina2014On}
M.~Deryagina and I.~Mednykh.
\newblock On the {J}acobian group for {M}\"obius ladder and prism graphs.
\newblock In {\em Geometry, integrability and quantization {XV}}, pages
  117--126. Avangard Prima, Sofia, 2014.

\bibitem{Dhar1990Self}
D.~Dhar.
\newblock Self-organized critical state of sandpile automaton models.
\newblock {\em Phys. Rev. Lett.}, 64(14):1613--1616, 1990.

\bibitem{Ducey2017On}
J.~E. Ducey.
\newblock On the critical group of the missing {M}oore graph.
\newblock {\em Discrete Math.}, 340(5):1104--1109, 2017.

\bibitem{Ducey2018The}
J.~E. Ducey, I.~Hill, and P.~Sin.
\newblock The critical group of the {K}neser graph on 2-subsets of an
  {$n$}-element set.
\newblock {\em Linear Algebra Appl.}, 546:154--168, 2018.

\bibitem{Glass2017Critical}
D.~B. Glass.
\newblock Critical groups of graphs with dihedral actions {II}.
\newblock {\em European J. Combin.}, 61:25--46, 2017.

\bibitem{Glass2014Critical}
D.~B. Glass and C.~Merino.
\newblock Critical groups of graphs with dihedral actions.
\newblock {\em European J. Combin.}, 39:95--112, 2014.

\bibitem{GOEL2019138}
G.~Goel and D.~Perkinson.
\newblock Critical groups of iterated cones.
\newblock {\em Linear Algebra and its Applications}, 567:138 -- 142, 2019.

\bibitem{Hou2006On}
Y.~Hou, C.~Woo, and P.~Chen.
\newblock On the sandpile group of the square cycle {$C_n^2$}.
\newblock {\em Linear Algebra Appl.}, 418(2-3):457--467, 2006.

\bibitem{Jacobson2003Critical}
B.~Jacobson, A.~Niedermaier, and V.~Reiner.
\newblock Critical groups for complete multipartite graphs and {C}artesian
  products of complete graphs.
\newblock {\em J. Graph Theory}, 44(3):231--250, 2003.

\bibitem{Krepkiy2013The}
I.~A. Krepkiy.
\newblock The sandpile groups of chain-cyclic graphs.
\newblock {\em Zap. Nauchn. Sem. S.-Peterburg. Otdel. Mat. Inst. Steklov.
  (POMI)}, 421(Teoriya Predstavleni\u\i , Dinamicheskie Sistemy, Kombinatornye
  Metody. XXIII):94--112, 2014.

\bibitem{Levine2011Sandpile}
L.~Levine.
\newblock Sandpile groups and spanning trees of directed line graphs.
\newblock {\em J. Combin. Theory Ser. A}, 118(2):350--364, 2011.

\bibitem{Lorenzini1989Arithmetical}
D.~J. Lorenzini.
\newblock Arithmetical graphs.
\newblock {\em Math. Ann.}, 285(3):481--501, 1989.

\bibitem{Musiker2009The}
G.~Musiker.
\newblock The critical groups of a family of graphs and elliptic curves over
  finite fields.
\newblock {\em J. Algebraic Combin.}, 30(2):255--276, 2009.

\bibitem{Raza2015On}
Z.~Raza.
\newblock On the critical group of certain subdivided wheel graphs.
\newblock {\em Punjab Univ. J. Math. (Lahore)}, 47(2):57--64, 2015.

\bibitem{Reiner2014Critical}
V.~Reiner and D.~Tseng.
\newblock Critical groups of covering, voltage and signed graphs.
\newblock {\em Discrete Math.}, 318:10--40, 2014.

\bibitem{Shi2011The}
W.-N. Shi, Y.-L. Pan, and J.~Wang.
\newblock The critical groups for {$K_m\vee P_n$} and {$P_m\vee P_n$}.
\newblock {\em Australas. J. Combin.}, 50:113--125, 2011.

\bibitem{Toumpakari2007On}
E.~Toumpakari.
\newblock On the sandpile group of regular trees.
\newblock {\em European J. Combin.}, 28(3):822--842, 2007.

\bibitem{NCG00966}
Y.~Zhou and H.~Chen.
\newblock The sandpile group of a family of nearly complete graphs.
\newblock {\em Bulletin of the Malaysian Mathematical Sciences Society}, pages
  1--13, 2020.

\end{thebibliography}

\end{document}